\documentclass[12pt,reqno]{amsart}

\usepackage{amssymb}

\usepackage{mathrsfs}

\DeclareMathAlphabet{\mathpzc}{OT1}{pzc}{m}{it}

\usepackage[all]{xy}

\entrymodifiers={+!!<0pt,\fontdimen22\textfont2>}

\font\sss=cmss8

\def\BV{{\mathbb V}}

\def\BX{{\mathbb X}}
\def\BY{{\mathbb Y}}
\def\BZ{{\mathbb Z}}

\def\sA{\mbox{\sf A}}

\def\sD{\mbox{\sf D}}

\def\sT{\mbox{\sf T}}
\def\sU{\mbox{\sf U}}
\def\sV{\mbox{\sf V}}

\def\ssA{\mbox{\sss A}}

\def\ssT{\mbox{\sss T}}

\def\add{\operatorname{add}}

\def\adots{\mathinner{\mkern1mu\raise1.0pt\vbox{\kern7.0pt\hbox{.}}\mkern2mu\raise4.0pt\hbox{.}\mkern2mu\raise7.0pt\hbox{.}\mkern1mu}}

\def\D{\mbox{\sD}}

\def\dim{\operatorname{dim}}

\def\End{\operatorname{End}}

\def\Hom{\operatorname{Hom}}

\def\Irr{\operatorname{Irr}}

\def\modulo{\operatorname{mod}}

\def\prod{\operatorname{prod}}

\def\rad{\operatorname{rad}}

\numberwithin{equation}{part}

\newtheorem{Lemma}{Lemma}[section]
\newtheorem{Theorem}[Lemma]{Theorem}
\newtheorem{Proposition}[Lemma]{Proposition}
\newtheorem{Corollary}[Lemma]{Corollary}

\theoremstyle{definition}

\newtheorem{Setup}[Lemma]{Setup}

\newtheorem{Remark}[Lemma]{Remark}

\begin{document}

\setlength{\parindent}{0pt}
\setlength{\parskip}{7pt}

\title[Cluster quotients]
{Quotients of cluster categories}

\author{Peter J\o rgensen}
\address{School of Mathematics and Statistics,
Newcastle University, Newcastle upon Tyne NE1 7RU,
United Kingdom}
\email{peter.jorgensen@ncl.ac.uk}
\urladdr{http://www.staff.ncl.ac.uk/peter.jorgensen}





\begin{abstract} 
  
Higher cluster categories were recently introduced as a
generalization of cluster categories.

\medskip
\noindent
This paper shows that in Dynkin types $A$ and $D$, half of all higher
cluster categories are actually just quotients of cluster categories.
The other half can be obtained as quotients of $2$-cluster categories,
the ``lowest'' type of higher cluster categories.

\medskip
\noindent
Hence, in Dynkin types $A$ and $D$, all higher cluster phenomena are
implicit in cluster categories and $2$-cluster categories.
In contrast, the same is not true in Dynkin type $E$.

\end{abstract}

\maketitle

\setcounter{section}{-1}
\section{Introduction}
\label{sec:introduction}

This paper is about the connection between quotient categories and
cluster categories, so let me start by explaining these two notions.

{\bf Quotient categories} come in a number of different flavours.  The
one to be considered here is probably the most basic: Let $\sA$ be an
additive category with a class of objects $\BY$.  For objects $A$ and
$B$ of $\sA$, denote by $\BY(A,B)$ all the morphisms from $A$ to $B$
which factor through an object of $\BY$.  Then the quotient category
$\sA_{\BY}$ has the same objects as $\sA$, and its morphism spaces are
defined by
\[
  \sA_{\BY}(A,B) = \sA(A,B)/\BY(A,B).
\]

{\bf Cluster categories} and the more general $u$-cluster categories,
pa\-ra\-me\-tri\-zed by the positive integer $u$, were introduced in
\cite{BMRRT}, \cite{CCS2}, \cite{Keller}, \cite{Thomas}, and
\cite{BinZhu}.  They underlie a representation theoretical viewpoint
on the theory of cluster algebras introduced and developed in
\cite{ClustAlgIII}, \cite{ClustAlgI}, \cite{ClustAlgII}, and
\cite{ClustAlgIV}.  Cluster categories have generated a strong
activity in recent years, not least because of their connection to
finite dimensional algebras and tilting theory, see \cite{Amiot},
\cite{ABST_cluster}, \cite{ABST_mcluster}, \cite{BaurMarsh1},
\cite{BaurMarsh2}, \cite{BMRRT}, \cite{BMR1}, \cite{BMR2},
\cite{BMR3}, \cite{BMRT}, \cite{CC}, \cite{CCS1}, \cite{CCS2},
\cite{CK1}, \cite{CK2}, \cite{GLS}, \cite{GLS2},
\cite{HolmJorgensenA}, \cite{HolmJorgensenD}, \cite{IyamaYoshino},
\cite{Keller}, \cite{KellerReiten}, \cite{KellerReiten2},
\cite{KoenigZhu}, \cite{Tabuada}, \cite{Thomas}, and \cite{BinZhu}.
One particular aspect of this connection is the importance of quivers
in both subjects, and in particular of Dynkin quivers of types $A$,
$D$, and $E$.

{\bf The present paper} shows that quotient categories permit a bridge
between cluster categories and $u$-cluster categories.  Specifically,
it will be proved that in Dynkin types $A$ and $D$, half of all
$u$-cluster categories are actually just quotients of cluster
categories, in a sense of the word ``half'' which will be made precise.

In the language of $u$-cluster categories, a cluster category is the
same thing as a $1$-cluster category, so an equivalent way to
formulate the statement is to say that, in types $A$ and $D$, half of
all $u$-cluster ca\-te\-go\-ri\-es are quotients of $1$-cluster
categories.  This will be complemented by a proof that the other half
of the $u$-cluster categories are quotients of $2$-cluster categories.
These statements are shown in Corollaries \ref{cor:A} and \ref{cor:D}.
In contrast, the corresponding result in type $E$ is not true, see
Remark \ref{rmk:E}.

The results will be proved in the following strong sense: It will be
established that the quotient categories in question are triangulated
categories which are equivalent {\em as triangulated categories} to
the relevant $u$-cluster categories.

As a backdrop to this, I will show more generally, under some
technical assumptions, that if $\sT$ is a triangulated category with a
class of objects $\BX$, then the quotient category $\sT_{\BX}$ is
triangulated if and only if $\BX$ is equal to its image under the
Auslander-Reiten translation of $\sT$ (Theorem
\ref{thm:TX_triangulated}).  When this is the case, I will prove that
as a translation quiver, the Auslander-Reiten quiver of $\sT_{\BX}$
can be obtained from the Auslander-Reiten quiver of $\sT$ by deleting
the vertices corresponding to objects in $\BX$ (Theorem
\ref{thm:AR}).  These results permit the deletion of orbits of the
Auslander-Reiten translation from the Auslander-Reiten quiver of a
triangulated category, without destroying the property of being
triangulated. 

Note that, while it is classical that Verdier quotients of
triangulated categories are triangulated, it seems to be less well
known that the present, simpler type of quotient categories sometimes
have a triangulated structure.  Some references do exist, see
\cite[sec.\ 7]{Beligiannis} and \cite[sec.  4]{IyamaYoshino}, but no
systematic exploration seems to have taken place.

The paper is organized as follows: Sections \ref{sec:pretriangulated}
and \ref{sec:triangulated} develop the theory of triangulated
quotients of triangulated categories.  Section \ref{sec:AR}
con\-si\-ders the Auslander-Reiten theory of triangulated quotients.
And Section \ref{sec:clusters} applies the theory to $u$-cluster
categories.

{\bf Background material. } Let me round off the introduction with
some facts about Krull-Schmidt categories and their quotients.

Let $k$ be an algebraically closed field and let $\sA$ be a
$k$-linear category which is Krull-Schmidt, that is, each object of
$\sA$ is a sum of indecomposable objects which are unique up to
isomorphism.

The radical $\rad$ of $\sA$ determines the subquotient
\[
  \Irr(M,N) = \rad(M,N)/\rad^2(M,N),
\]
see \cite[p.\ 178 and p.\ 228]{ARS}.  The Auslander-Reiten (AR) quiver
of $\sA$ has one vertex for each isomorphism class of indecomposable
objects and $\dim_k \Irr(M,N)$ arrows from the vertex of $M$ to the
vertex of $N$.

Now let $\BY$ be a class of objects of $\sA$ closed under
isomorphisms, direct sums, and direct summands.  Then the quotient
category $\sA_{\BY}$ makes sense, and the useful content of the
following lemma is folklore.
\begin{Lemma}
\label{lem:indecomposables_in_factor}

\begin{enumerate}
  
  \item  If $A \cong B$ in $\sA_{\BY}$, then there exist objects $X,Y$ in
  $\BY$ such that in $\sA$, the object $A$ is isomorphic to a direct
  summand of $B \oplus Y$ and the object $B$ is isomorphic to a direct
  summand of $A \oplus X$.

\smallskip

  \item  The category $\sA_{\BY}$ is Krull-Schmidt and the isomorphism
  classes of indecomposable objects in $\sA_{\BY}$ correspond to the
  isomorphism classes of indecomposable objects in $\sA$ which are not
  in $\BY$.

\smallskip

  \item  The AR quiver of $\sA_{\BY}$ is obtained from the AR quiver of
  $\sA$ by deleting the vertices corresponding to objects of $\BY$
  along with the arrows into or out of such vertices.

\end{enumerate}
\end{Lemma}

(i) and (ii) are straightforward, with (ii) following from (i).

(iii) is obtained by combining (ii) and the following: If $M$ and $N$
are indecomposable objects of $\sA$ which are not in $\BY$, then
$\rad_{\ssA_{\BY}}^n(M,N) \cong \rad_{\ssA}^n(M,N)/\BY(M,N)$ whence
\[
  \Irr_{\ssA_{\BY}}(M,N) \cong \Irr_{\ssA}(M,N).
\]

\section{Pretriangulated quotient categories}
\label{sec:pretriangulated}

As a stepping stone towards triangulated quotient categories, this
section shows how to equip quotient categories with a pretriangulated
structure.

Recall from \cite[sec.\ II.1]{BeligiannisReiten} the notion of a
pretriangulated category.  This is an additive category equipped with
some data: There is an endofunctor $\sigma$ and a class of diagrams of
the form $A \rightarrow B \rightarrow C \rightarrow \sigma A$ called
distinguished right-triangles.  There is also an endofunctor $\omega$
and a class of diagrams of the form $\omega Z \rightarrow X
\rightarrow Y \rightarrow Z$ called distinguished left-triangles.  The
distinguished right-triangles satisfy the axioms of a triangulated
category, except that $\sigma$ is not required to be an equivalence of
categories.  Similarly for the distinguished left-triangles.  Finally,
there are some compatibility conditions, most importantly that
$(\sigma,\omega)$ is an adjoint pair of functors.

\begin{Setup}
\label{set:preliminary}
Let $\sT$ be a triangulated category with suspension functor $\Sigma$.
Let $\BX$ be a class of objects of $\sT$, closed under isomorphisms,
which is both preenveloping and precovering.
\end{Setup}

Recall that for $\BX$ to be preenveloping means that each object $M$
has an $\BX$-pre\-en\-ve\-lo\-pe, that is a morphism $M \rightarrow
X_M$ with $X_M$ in $\BX$ such that each morphism $M \rightarrow X$
with $X$ in $\BX$ factors through $M \rightarrow X_M$,
\begin{equation}
\label{equ:amper}
\vcenter{
  \xymatrix{
    M \ar[r] \ar[dr] & X_M \ar@{.>}[d] \\
    & X \lefteqn{.} \\
           }
        }
\end{equation}
Dually, for $\BX$ to be precovering means that each object $M$ has an
$\BX$-precover $X^M \rightarrow M$,
\[
  \xymatrix{
    X \ar[rd] \ar@{.>}[d] & \\
    X^M \ar[r] & M \lefteqn{.} \\
           }
\]

Under Setup \ref{set:preliminary}, the quotient category $\sT_{\BX}$
can be turned into a pretriangulated category as follows:

First, to get the endofunctors $\sigma$ and $\omega$, pick, for each
$M$ in $\sT$, an $\BX$-preenvelope $M \rightarrow X_M$ and an
$\BX$-precover $X^M \rightarrow M$.  Complete to distinguished
triangles in $\sT$,
\begin{equation}
\label{equ:defsigma}
  M \rightarrow X_M \rightarrow \sigma M \rightarrow \Sigma M
\end{equation}
and
\begin{equation}
\label{equ:defomega}
  \Sigma^{-1}M \rightarrow \omega M \rightarrow X^M \rightarrow M.
\end{equation}
This defines objects $\sigma M$ and $\omega M$.  To turn $\sigma$ and
$\omega$ into functors, note that if $M \stackrel{\mu}{\rightarrow} N$
is a morphism in $\sT$, then there is a commutative diagram
\begin{equation}
\label{equ:a}
  \vcenter{
  \xymatrix{
    M \ar[d]_{\mu} \ar[r] & X_M \ar[d]_{\xi} \ar[r] & \sigma M \ar[d]^s \ar[r]
      & \Sigma M \ar[d]^{\Sigma \mu} \\  
    N \ar[r] & X_N \ar[r] & \sigma N \ar[r] & \Sigma N \lefteqn{,}
           }
          }
\end{equation}
where $\xi$ exists because $X_N$ is in $\BX$ and $M \rightarrow X_M$
is an $\BX$-preenvelope, and $s$ exists by one of the axioms for the
triangulated category $\sT$.  Now denote by $\underline{\mu}$ and
$\underline{s}$ the morphisms in $\sT_{\BX}$ corresponding to $\mu$
and $s$ and set $\sigma(\underline{\mu}) = \underline{s}$.  This turns
$\sigma$ into an endofunctor of $\sT_{\BX}$, and the dual method works
for $\omega$.

Secondly, to get distinguished right- and left-triangles, let $M
\stackrel{\mu}{\rightarrow} N$ be an $\BX$-monomorphism in $\sT$, that
is, a morphism such that each morphism $M \rightarrow X$ with $X$ in
$\BX$ factors through $M \stackrel{\mu}{\rightarrow} N$,
\begin{equation}
\label{equ:X-monomorphism}
  \vcenter{
  \xymatrix{
    M \ar[r]^{\mu} \ar[dr] & N \ar@{.>}[d] \\
    & X \lefteqn{.} \\
           }
          }
\end{equation}
Then $\mu$ can be extended to a distinguished triangle $M
\stackrel{\mu}{\rightarrow} N \stackrel{\nu}{\rightarrow} P
\rightarrow \Sigma M$ which fits into a commutative diagram
\begin{equation}
\label{equ:b}
  \vcenter{
  \xymatrix{
    M \ar@{=}[d] \ar[r]^{\mu} & N \ar[d]_n \ar[r]^{\nu} & P \ar[d]^{\pi} \ar[r]
      & \Sigma M \ar@{=}[d] \\  
    M \ar[r] & X_M \ar[r] & \sigma M \ar[r] & \Sigma M \lefteqn{,}
           }
          }
\end{equation}
where $n$ exists because $X_M$ is in $\BX$ and $M
\stackrel{\mu}{\rightarrow} N$ is an $\BX$-monomorphism, and $\pi$
exists by one of the axioms for the triangulated category $\sT$.
Declare 
\[
  M \stackrel{\underline{\mu}}{\rightarrow}
  N \stackrel{\underline{\nu}}{\rightarrow}
  P \stackrel{\underline{\pi}}{\rightarrow} \sigma M
\]
to be a distinguished right-triangle in $\sT_{\BX}$.

Dually, let $M \rightarrow N$ be an $\BX$-epimorphism,
\begin{equation}
\label{equ:X-epimorphism}
  \vcenter{
  \xymatrix{
    X \ar[dr] \ar@{.>}[d] & \\
    M \ar[r] & N \lefteqn{.} \\
           }
          }
\end{equation}
Then $M \rightarrow N$ can be extended to a distinguished triangle,
and proceeding dually to the above construction gives the
distinguished left-triangles in $\sT_{\BX}$.

With this data, the following theorem holds.

\begin{Theorem}
The quotient category $\sT_{\BX}$ is pretriangulated.
\end{Theorem}

\begin{proof}
Morally speaking, this is \cite[thm.\ 7.2]{Beligiannis}, which,
however, is stated with different assumptions.  The proof can be
carried through in the manner of the proof of \cite[thm.\
3.1]{BeligiannisMarmaridis}.

It is helpful to start by establishing that the following construction
is an alternative way of getting the distinguished right-triangles in
$\sT_{\BX}$: Given a morphism $M \stackrel{\mu}{\rightarrow} N$ in
$\sT$, consider the distinguished triangle $M \rightarrow X_M
\rightarrow \sigma M \rightarrow \Sigma M$ from equation
\eqref{equ:defsigma}.  The octahedral axiom for $\sT$ gives a way to
embed it into a commutative diagram where the rows are distinguished
triangles,
\[
  \xymatrix{
    M \ar[r] \ar[d]_{\mu} & X_M \ar[r] \ar[d] & \sigma M \ar[r] \ar@{=}[d] & \Sigma M \ar[d]^{\Sigma \mu} \\
    N \ar[r]_{\nu} & P \ar[r]_{\pi} & \sigma M \ar[r] & \Sigma N \lefteqn{.} \\
           }
\]
Up to isomorphism, the distinguished right-triangles in $\sT_{\BX}$
are now precisely the diagrams which can be obtained as
\[
  M \stackrel{\underline{\mu}}{\rightarrow}
  N \stackrel{\underline{\nu}}{\rightarrow}
  P \stackrel{-\underline{\pi}}{\rightarrow}
  \sigma M.
\]
Note the sign change on the last arrow.
\end{proof}

\begin{Remark}
Standard arguments show that none of the choices involved in
constructing the pretriangulated structure on $\sT_{\BX}$ make any
difference.

That is, choosing the distinguished triangles \eqref{equ:defsigma} and
\eqref{equ:defomega} or the morphisms in the diagrams \eqref{equ:a}
and \eqref{equ:b} differently gives an equivalent structure of
pretriangulated category.
\end{Remark}

\section{Triangulated quotient categories}
\label{sec:triangulated}

This section considers the pretriangulated quotient category
$\sT_{\BX}$ from Section \ref{sec:pretriangulated} and shows, under
some technical assumptions, that it is triangulated if and only if
$\BX$ is equal to its image under the Auslander-Reiten (AR)
translation of $\sT$.

\begin{Setup}
\label{set:standard}
Let $k$ be an algebraically closed field and let $\sT$ be a $k$-linear
triangulated ca\-te\-go\-ry with suspension functor $\Sigma$ and Serre
functor $S$, which has finite dimensional $\Hom$ spaces and split
idempotents.  Let $\BX$ be a class of objects of $\sT$, closed under
isomorphisms, direct sums, and direct summands, which is both
precovering and preenveloping.
\end{Setup}

The conditions imply that $\sT$ is Krull-Schmidt, see \cite[p.\ 
52]{RingTame}.  Recall that the Serre functor $S$ is an
autoequivalence of $\sT$ for which there are natural isomorphisms
\[
  \sT(A,B) \cong \sT(B,SA)^{\vee},
\]
where $(-)^{\vee} = \Hom_k(-,k)$.  The Serre functor $S$ determines
the AR translation $\tau = \Sigma^{-1} \circ S$.

For the following lemma, note that I will write ``$\tau \BX = \BX$''
as a shorthand for ``the set of objects isomorphic to objects in $\tau
\BX$ is equal to $\BX$''.  The notions of $\BX$-monomorphism and
$\BX$-epimorphism are described by diagrams \eqref{equ:X-monomorphism}
and \eqref{equ:X-epimorphism}.

\begin{Lemma}
\label{lem:mono_epi}
Suppose $\tau \BX = \BX$.  If $A \rightarrow B \rightarrow C
\rightarrow$ is a distinguished triangle in $\sT$, then $A \rightarrow
B$ is an $\BX$-monomorphism if and only if $B \rightarrow C$ is an
$\BX$-epimorphism.
\end{Lemma}

\begin{proof}
Let $X$ run through $\BX$.  For each $X$, the distinguished triangle
induces a long exact sequence
\[
  \sT(X,B) \rightarrow \sT(X,C) \rightarrow \sT(X,\Sigma A)
  \rightarrow \sT(X,\Sigma B).
\]
For $B \rightarrow C$ to be an $\BX$-epimorphism is the same as for
the first arrow in the long exact sequence always to be surjective.
This is the same as for the second arrow always to be zero, which is
again the same as for the third arrow always to be injective.

Using Serre duality, the third arrow can be identified with
\[
  \sT(A,\Sigma^{-1}SX)^{\vee} \rightarrow \sT(B,\Sigma^{-1}SX)^{\vee}
\]
which is injective if and only if
\[
  \sT(B,\Sigma^{-1}SX) \rightarrow \sT(A,\Sigma^{-1}SX)
\]
is surjective.  But this last arrow is
\[
  \sT(B,\tau X) \rightarrow \sT(A,\tau X).
\]
For this always to be surjective is the same as for $A \rightarrow B$
to be a $(\tau\BX)$-monomorphism, that is, an $\BX$-monomorphism.
\end{proof}

The pretriangulated structure of $\sT_{\BX}$ is a triangulated
structure if and only if the functor $\sigma$ is an autoequivalence.

\begin{Theorem}
\label{thm:TX_triangulated}
The pretriangulated structure of $\sT_{\BX}$ is a triangulated
structure if and only if $\tau \BX = \BX$.
\end{Theorem}

\begin{proof}
First assume that $\tau \BX = \BX$; I must show that $\sigma$ is an
au\-to\-e\-qui\-va\-len\-ce. 

It follows directly from the definitions (diagrams \eqref{equ:amper}
and \eqref{equ:X-monomorphism}) that an $\BX$-preenvelope is simply an
$\BX$-monomorphism $M \rightarrow X$ with $X$ in $\BX$; similarly, an
$\BX$-precover is an $\BX$-epimorphism $X \rightarrow M$ with $X$ in
$\BX$.  Consider the distinguished triangle \eqref{equ:defsigma},
\begin{equation}
\label{equ:defsigma_again}
  M \rightarrow X_M \rightarrow \sigma M \rightarrow \Sigma M.
\end{equation}
The morphism $M \rightarrow X_M$ is an $\BX$-preenvelope and hence an
$\BX$-mo\-no\-mor\-phism, and so by Lemma \ref{lem:mono_epi} the
morphism $X_M \rightarrow \sigma M$ is an $\BX$-epimorphism and hence
an $\BX$-precover.  Now, $\omega \sigma M$ is computed by completing
such a precover to a distinguished triangle as in equation
\eqref{equ:defomega}, but completing $X_M \rightarrow \sigma M$ to a
distinguished triangle just recovers \eqref{equ:defsigma_again}, so
$\omega \sigma M \cong M$ in $\sT_{\BX}$.

This isomorphism is easily shown to be natural.  Similarly, there is a
natural isomorphism $\sigma \omega M \cong M$, and so $\sigma$ is an
autoequivalence whence $\sT_{\BX}$ is triangulated.

Next assume that $\sT_{\BX}$ is triangulated, that is, $\sigma$ (and
hence also $\omega$) is an autoequivalence; I must show $\tau \BX =
\BX$.  This amounts to seeing $\tau \BX \subseteq \BX$ and
$\tau^{-1}\BX \subseteq \BX$, and I will give the proof of the first
of these since the second one is dual.

Let $X$ be an indecomposable object in $\BX$ and consider the AR
triangle
\begin{equation}
\label{equ:c}
  \tau X \stackrel{t}{\rightarrow} Y \rightarrow X \rightarrow
\end{equation}
in $\sT$.  If $\tau X$ is in $\BX$ then I am done, so let me suppose
not.  When $\tau X$ is not in $\BX$, a morphism in $\sT$ from $\tau X$
to an object in $\BX$ cannot be a split monomorphism, so any such
morphism factors through $t$ which is hence an $\BX$-monomorphism.
This means that there is a distinguished right-triangle
\[
  \tau X
  \stackrel{\underline{t}}{\rightarrow} Y
  \rightarrow X
  \rightarrow \sigma \tau X
\]
in $\sT_{\BX}$, and this is isomorphic to
\[
  \tau X \stackrel{\underline{t}}{\rightarrow} Y
  \rightarrow 0 \rightarrow \sigma \tau X.
\]

Since $\sT_{\BX}$ is triangulated, this shows that $\underline{t}$ is
an isomorphism, and Lemma \ref{lem:indecomposables_in_factor}(i) gives
that $Y$ is a direct summand of $\tau X \coprod X^{\prime}$ in $\sT$
for some $X^{\prime}$ in $\BX$.  Hence $Y$ has the form $\tau X
\coprod X_1$ for some $X_1$ in $\BX$.  Note that $\tau X$ is forced to
be among the indecomposable direct summands of $Y$, since $Y$ would
otherwise be zero in $\sT_{\BX}$ forcing $\tau X$ to be zero in
$\sT_{\BX}$ and thereby contradicting that $\tau X$ is not in $\BX$.

But now $t$ has the form
\[
  \tau X \stackrel{
                   \left( \begin{array}{c} \scriptstyle u \\ \scriptstyle v \end{array} \right)
                  }{\longrightarrow}
  \tau X \coprod X_1,
\]
and since $X_1$ is in $\BX$ and hence zero in $\sT_{\BX}$, the
morphism $\underline{t}$ equals $\underline{u}$.  However,
$\underline{t}$ and hence $\underline{u}$ is an isomorphism, and so an
invertible element of the ring $\End_{\ssT_{\BX}}(\tau X)$.  This ring
is a quotient of $\End_{\ssT}(\tau X)$, and $\underline{u}$ is the
image of $u$.  But $\tau X$ is indecomposable so $\End_{\ssT}(\tau X)$
is local, and it follows that since $\underline{u}$ is invertible in
the quotient, $u$ must itself be invertible.

This implies that $t$ is split, contradicting that \eqref{equ:c} is an
AR triangle.
\end{proof}

\section{The Auslander-Reiten theory of quotient categories}
\label{sec:AR}

This section continues to work under Setup \ref{set:standard}.  

The quotient category $\sT_{\BX}$ is Krull-Schmidt by Lemma
\ref{lem:indecomposables_in_factor}(ii), and if $\tau\BX = \BX$ then
$\sT_{\BX}$ is triangulated by Theorem \ref{thm:TX_triangulated}.
This section shows that when $\sT_{\BX}$ is triangulated then it has
AR triangles and its AR quiver can be computed, as a translation
quiver, from the AR quiver of $\sT$ by deleting the vertices
corresponding to objects of $\BX$ along with the arrows into or out of
such vertices.  This sharpens Lemma
\ref{lem:indecomposables_in_factor}(iii).

\begin{Proposition}
\label{pro:AR}
Suppose that the AR translation $\tau$ of $\sT$ satisfies $\tau \BX =
\BX$.

Then the triangulated category $\sT_{\BX}$ has a Serre functor
$\underline{S}$, and the AR 
translation $\sigma^{-1} \circ \underline{S}$ of $\sT_{\BX}$ is equal
to the functor $\underline{\tau}$ on $\sT_{\BX}$ which is induced by
the functor $\tau$ on $\sT$.
\end{Proposition}

\begin{proof}
Since $\sT_{\BX}$ is triangulated, $\omega$ is e\-qui\-va\-lent to
$\sigma^{-1}$ and the distinguished triangle \eqref{equ:defomega}
from Section \ref{sec:pretriangulated} reads
\begin{equation}
\label{equ:defsigmainverse}
  \Sigma^{-1}M \rightarrow \sigma^{-1}M \rightarrow X^M \rightarrow M.
\end{equation}

Let $N$ be an object of $\sT_{\BX}$ and note that $N$ can also be
viewed as an object of $\sT$.  Since $X^M \rightarrow M$ is an
$\BX$-precover, it is easy to see that the cokernel of $\sT(N,X^M)
\rightarrow \sT(N,M)$ is $\sT_{\BX}(N,M)$, and so there is an exact
sequence 
\begin{equation}
\label{equ:f}
  \sT(N,X^M) \rightarrow \sT(N,M) \rightarrow \sT_{\BX}(N,M) \rightarrow 0.
\end{equation}

For each $L$ in $\sT$, the distinguished triangle
\eqref{equ:defsigmainverse} gives a long exact sequence
\begin{equation}
\label{equ:defsigmainverse_long_exact_sequence}
  \sT(X^M,\Sigma^{-1}L) \rightarrow \sT(\sigma^{-1}M,\Sigma^{-1}L)
             \rightarrow \sT(M,L)
             \rightarrow \sT(X^M,L).
\end{equation}
Since $X^M \rightarrow M$ is an $\BX$-epimorphism, Lemma
\ref{lem:mono_epi} gives that $\sigma^{-1}M \rightarrow X^M$ is an
$\BX$-monomorphism and hence an $\BX$-preenvelope, and it is again
easy to see that the cokernel of
\[
  \sT(X^M,\Sigma^{-1}L) \rightarrow \sT(\sigma^{-1}M,\Sigma^{-1}L)
\]
is $\sT_{\BX}(\sigma^{-1}M,\Sigma^{-1}L)$.  However, by
\eqref{equ:defsigmainverse_long_exact_sequence} this co\-ker\-nel is
isomorphic to the kernel of
\[
  \sT(M, L) \rightarrow \sT(X^M,L),
\]
so
there is an exact sequence
\[
  0 \rightarrow \sT_{\BX}(\sigma^{-1}M,\Sigma^{-1}L)
    \rightarrow \sT(M,L)
    \rightarrow \sT(X^M,L).
\]

Setting $L = SN$ and taking the $k$-linear dual gives an exact
sequence 
\begin{equation}
\label{equ:g}
  \sT(X^M,SN)^{\vee} \rightarrow
  \sT(M,SN)^{\vee} \rightarrow
  \sT_{\BX}(\sigma^{-1}M,\Sigma^{-1}(SN))^{\vee} \rightarrow 0.
\end{equation}
But the sequences \eqref{equ:f} and \eqref{equ:g} fit together in a
commutative diagram
\[
  \xymatrix{
  \sT(N,X^M) \ar[r] \ar[d]^{\wr} & \sT(N,M) \ar[r] \ar[d]^{\wr} & \sT_{\BX}(N,M) \ar[r] \ar[d]^{\wr} & 0 \\
  \sT(X^M,SN)^{\vee} \ar[r] & \sT(M,SN)^{\vee} \ar[r] & \sT_{\BX}(\sigma^{-1}M,\Sigma^{-1}(SN))^{\vee} \ar[r] & 0 \lefteqn{,}
           }
\]
where the two first isomorphisms are by the definition of the Serre
functor $S$ and the third isomorphism follows from the first two.  So
there is a natural isomorphism
\[
  \sT_{\BX}(N,M) \cong \sT_{\BX}(M,\sigma(\Sigma^{-1}SN))^{\vee}.
\]

Hence $\sT_{\BX}$ has a right Serre functor $\underline{S}$ which is
given on objects by $\underline{S}N = \sigma(\Sigma^{-1}SN)$.  Similar
computations show that the same formula gives a left Serre functor,
and hence a Serre functor, see \cite[sec.\ I.1]{RVdB}.

The formula implies that the AR translation $\sigma^{-1} \circ
\underline{S}$ of $\sT_{\BX}$ is induced by $\Sigma^{-1}S = \tau$.
\end{proof}

Let $\Gamma$ be the AR quiver of $\sT$.  For the following theorem,
recall that the category $\sT$ is called standard if it is equivalent
to $\add k(\Gamma)$, the additive closure of the mesh category
$k(\Gamma)$, see \cite[def.\ 5.1]{BongartzGabriel}.

\begin{Theorem}
\label{thm:AR}
Suppose $\tau \BX = \BX$.

\begin{enumerate}

  \item  The triangulated category $\sT_{\BX}$ has AR triangles.

\smallskip
  
  \item As a translation quiver, the AR quiver $\Gamma^{\prime}$ of
  $\sT_{\BX}$ is obtained from the AR quiver $\Gamma$ of $\sT$ by
  deleting the vertices corresponding to objects of $\BX$ along with
  the arrows into or out of such vertices.

\smallskip

  \item  If $\sT$ is standard, then so is $\sT_{\BX}$.

\end{enumerate}
\end{Theorem}

\begin{proof}
(i) By Proposition \ref{pro:AR}, the category $\sT_{\BX}$ has a Serre
functor, and hence also AR triangles by \cite[prop.\ I.2.3]{RVdB}.

(ii) Lemma \ref{lem:indecomposables_in_factor}(iii) implies that the
vertices and arrows of $\Gamma^{\prime}$ can be obtained from those of
$\Gamma$ as described.  The AR translation of $\Gamma$ descends to the
AR translation of $\Gamma^{\prime}$ by Proposition \ref{pro:AR}.

(iii)  Suppose that $\sT$ is standard.  This means that there is an
equivalence of categories $\sT \simeq \add k(\Gamma)$.  Let $\BV$ be
the set of objects in $\add k(\Gamma)$ which corresponds to $\BX$, and
let $\Gamma^{\prime}$ be the AR quiver of $\sT_{\BX}$.  There is an
equivalence $\sT_{\BX} \simeq (\add k(\Gamma))_{\BV}$, so it is enough
to see that there is an equivalence $(\add k(\Gamma))_{\BV} \simeq
\add k(\Gamma^{\prime})$.

However, there is a functor $\add k(\Gamma) \rightarrow \add
k(\Gamma^{\prime})$ given by sending each object of $\BV$ to $0$.
This functor factors through a functor $(\add k(\Gamma))_{\BV}
\rightarrow \add k(\Gamma^{\prime})$ which is the required
equivalence. 
\end{proof}

\section{Cluster categories}
\label{sec:clusters}

This section applies the methods of the previous sections to show in
Corollaries \ref{cor:A} and \ref{cor:D} that in Dynkin types $A$ and
$D$, all $u$-cluster categories are triangulated quotients of $1$- and
$2$-cluster categories.

The following result is due to Amiot \cite[thm.\ 1.1.1 and thm.\ 7.0.5
plus proof]{Amiot}.

\begin{Theorem}
\label{thm:Amiot}
Let $\sU$ and $\sV$ be $k$-linear triangulated categories with finite
dimensional $\Hom$ spaces and split idempotents.  Suppose that $\sU$
and $\sV$ are of algebraic origin, standard, connected, and have only
a finite number of isomorphism classes of indecomposable objects.

Then $\sU$ and $\sV$ have AR triangles, and if the AR quivers of $\sU$
and $\sV$ are isomorphic as translation quivers, then $\sU$ and $\sV$
are equivalent as triangulated categories.
\end{Theorem}

Note that in particular, $\sU$ or $\sV$ could be taken to be a
$u$-cluster category of finite type.  Such categories are of algebraic
origin by the theory of \cite[sec.\ 9.3]{Keller} (see also \cite[sec.\ 
7.3]{Amiot}), and they are standard by \cite[prop.\ 6.1.1]{Amiot}.

\setcounter{subsection}{0}
\subsection{Type A}

\begin{Theorem}
\label{thm:A}
Let $u \geq v$ be positive integers for which $u \equiv v \modulo 2$.
Let $m,n$ be positive integers such that
\[
  u(m+1) = v(n+1).
\]
Then the $u$-cluster category of type $A_m$ is triangulated equivalent
to a quotient category of the $v$-cluster category of type $A_n$.
\end{Theorem}

\begin{proof}  
The proof will work by appealing to the theory of triangulated
quotient categories developed above, and to Theorem \ref{thm:Amiot}.
The categories which will be inserted into Theorem \ref{thm:Amiot}
will be (1) the $u$-cluster category of type $A_m$, and (2) a
suitably constructed quotient of the $v$-cluster category of type
$A_n$. As remarked above, the conditions of Theorem \ref{thm:Amiot}
hold for $u$-cluster categories of finite type; I will show along
the way that they also hold for the quotient category in question.

The AR quiver of $\D(kA_n)$ is $\BZ A_n$ by \cite[cor.\
4.5]{Happel}; let me visualize this as a horizontal band $n$ vertices
wide, 
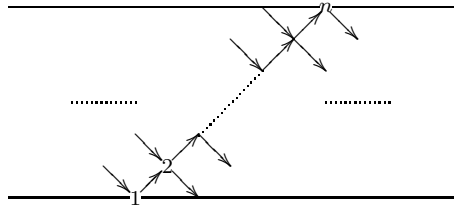
\[
  \def\objectstyle{\scriptstyle}
  \vcenter{
  \xymatrix @!0 @-1pc {
    *{} \ar@{-}[rrrrrrrr] & & & & & & & & *{} \ar[dr] \ar@{-}[rr] & & *{n} \ar[dr] \ar@{-}[rrrr] & & & & *{} \\
    & & & & & & & *{} \ar[dr] & & *{} \ar[ur] \ar[dr] & & *{} & & & \\
    *{} & & & & & & & & *{} \ar[ur] & & *{} & & & & *{} \\
    & & *{} \ar@{.}[rr] & & *{} & & & & & & *{} \ar@{.}[rr] & & *{} & & &\\
    *{} & & & & *{} \ar[dr] & & *{} \ar@{.}[uurr] \ar[dr] & & & & & && & \\
    & & & *{} \ar[dr] & & *{2} \ar[ur] \ar[dr] & & *{} & & & & & & & \\
    *{} \ar@{-}[rrrr] & & & & *{1} \ar[ur] \ar@{-}[rr] & & *{} \ar@{-}[rrrrrrrr] & & & & & & & & *{} \\
                      }
          }.
\]
Declare the distance between two horizontally neighbouring vertices to
be one unit.  The action of $\tau^{-1}$ on the AR quiver is to shift
one unit to the right.  The action of $\Sigma$ is to shift
$\frac{n+1}{2}$ units to the right, and reflect in the horizontal
centre line.  Both claims follow from \cite[table p.\
359]{MiyachiYekutieli}.  Hence the action of $\tau^{-1}\Sigma^v$ on
the AR quiver is to shift $v\frac{n+1}{2} + 1$ units to the right, and
reflect in the horizontal centre line if $v$ is odd.

Let $\sT$ denote the $v$-cluster category of type $A_n$, that is,
\[
  \sT = \D(kA_n)/\tau^{-1}\Sigma^v.
\]
The AR quiver of $\sT$ is the AR quiver of $\D(kA_n)$ modulo the
action of $\tau^{-1}\Sigma^v$ by \cite[prop.\ 1.3]{BMRRT}.  Hence the
AR quiver of $\sT$ is $\BZ A_n$ modulo an automorphism given by
shifting $v\frac{n+1}{2} + 1$ units to the right, and reflecting in
the horizontal centre line if $v$ is odd.

This means that if $v$ is even, then the AR quiver of $\sT$ has the
form 
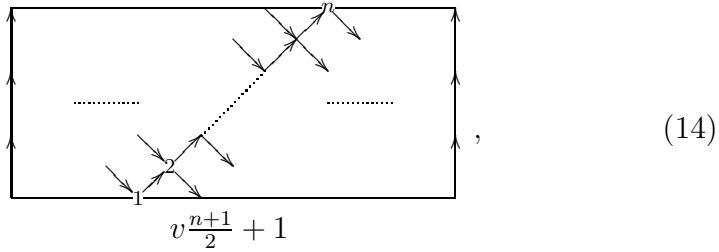
\begin{equation}
\label{equ:Aeven}
  \def\objectstyle{\scriptstyle}
  \vcenter{
  \xymatrix @!0 @-1pc {
    *{} \ar@{-}[rrrrrrrr] & & & & & & & & *{} \ar[dr] \ar@{-}[rr] & & *{n} \ar[dr] \ar@{-}[rrrr] & & & & *{} \\
    & & & & & & & *{} \ar[dr] & & *{} \ar[ur] \ar[dr] & & *{} & & & \\
    *{} \ar[uu] & & & & & & & & *{} \ar[ur] & & *{} & & & & *{} \ar[uu] \\
    & & *{} \ar@{.}[rr] & & *{} & & & & & & *{} \ar@{.}[rr] & & *{} & & \\
    *{} \ar[uu] & & & & *{} \ar[dr] & & *{} \ar@{.}[uurr] \ar[dr] & & & & & & & & *{}\ar[uu]\\
    & & & *{} \ar[dr] & & *{2} \ar[ur] \ar[dr] & & *{} & & & & & & & \\
    *{} \ar@{-}[rrrr] \ar[uu] & & & & *{1} \ar[ur] \ar@{-}[rr] & & *{} \ar@{-}[rrrrrrrr] & & & & & & & & *{} \ar[uu]\\
    & & & & & \lefteqn{\textstyle{v\frac{n+1}{2}+1}} & & & & & & & & & \\
                      }
          }\;,
\end{equation}
where the vertical ends of the rec\-tan\-gle are identified, with the
orientation indicated by the arrows.  The number below the quiver
indicates side length.  And if $v$ is odd, then the AR quiver of $\sT$
has the form 
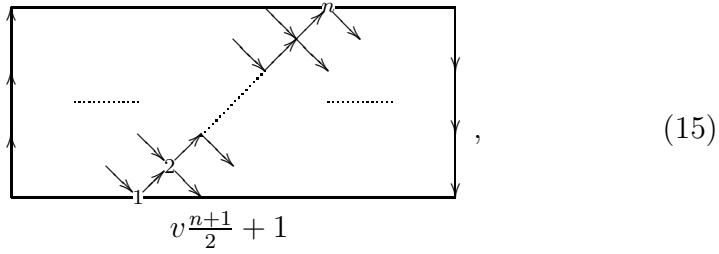
\begin{equation}
\label{equ:Aodd}
  \def\objectstyle{\scriptstyle}
  \vcenter{
  \xymatrix @!0 @-1pc {
    *{} \ar@{-}[rrrrrrrr] & & & & & & & & *{} \ar[dr] \ar@{-}[rr] & & *{n} \ar[dr] \ar@{-}[rrrr] & & & & *{} \ar[dd]\\
    & & & & & & & *{} \ar[dr] & & *{} \ar[ur] \ar[dr] & & *{} & & & \\
    *{} \ar[uu] & & & & & & & & *{} \ar[ur] & & *{} & & & & *{} \ar[dd] \\
    & & *{} \ar@{.}[rr] & & *{} & & & & & & *{} \ar@{.}[rr] & & *{} & & \\
    *{} \ar[uu] & & & & *{} \ar[dr] & & *{} \ar@{.}[uurr] \ar[dr] & & & & & & & & *{}\ar[dd]\\
    & & & *{} \ar[dr] & & *{2} \ar[ur] \ar[dr] & & *{} & & & & & & & \\
    *{} \ar@{-}[rrrr] \ar[uu] & & & & *{1} \ar[ur] \ar@{-}[rr] & & *{} \ar@{-}[rrrrrrrr] & & & & & & & & *{} \\
    & & & & & \lefteqn{\textstyle{v\frac{n+1}{2}+1}} & & & & & & & & & \\
                      }
          }\;,
\end{equation}
where the vertical ends of the rec\-tan\-gle are again identified,
with the orientation indicated by the arrows.

Let me now prove the theorem for $u$ and $v$ even.  Here the AR quiver
of $\sT$ is given by diagram \eqref{equ:Aeven}.  Consider the
indecomposable objects in a band $n-m$ vertices wide along the bottom
of the AR quiver, and let $\BX$ denote $\add$ of these.  Since $\BX$
is $\add$ of a finite set of objects, it is precovering and
preenveloping, and since $\tau$ preserves the set of objects in
question, it is clear that $\tau \BX = \BX$.

Let me check the conditions of Theorem \ref{thm:Amiot} for the
quotient category $\sT_{\BX}$.  Theorem \ref{thm:TX_triangulated} says
that it is triangulated.  It is clear that $\sT_{\BX}$ is $k$-linear
and has finite dimensional $\Hom$ spaces.  It is an exercise to check
that $\sT_{\BX}$ has split idempotents.  Since $\sT$ is of algebraic
origin, so is $\sT_{\BX}$.  Since $\sT$ is standard, $\sT_{\BX}$ is
standard by Theorem \ref{thm:AR}(iii).

Finally, by Theorem \ref{thm:AR}(ii), the AR quiver of $\sT_{\BX}$ is
obtained by deleting from the AR quiver of $\sT$ the band $n-m$
vertices wide along the bottom.  Therefore the AR quiver of
$\sT_{\BX}$ is $m$ vertices wide and has the form
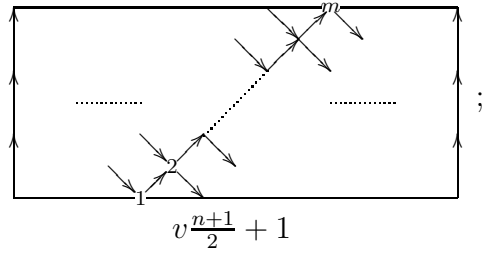
\[
  \def\objectstyle{\scriptstyle}
  \def\labelstyle{\textstyle}
  \xymatrix @!0 @-1pc {
    *{} \ar@{-}[rrrrrrrr] & & & & & & & & *{} \ar[dr] \ar@{-}[rr] & & *{m} \ar[dr] \ar@{-}[rrrr] & & & & *{} \\
    & & & & & & & *{} \ar[dr] & & *{} \ar[ur] \ar[dr] & & *{} & & & \\
    *{} \ar[uu] & & & & & & & & *{} \ar[ur] & & *{} & & & & *{} \ar[uu] \\
    & & *{} \ar@{.}[rr] & & *{} & & & & & & *{} \ar@{.}[rr] & & *{} & & \\
    *{} \ar[uu] & & & & *{} \ar[dr] & & *{} \ar@{.}[uurr] \ar[dr] & & & & & & & & *{}\ar[uu]_{\;;}\\
    & & & *{} \ar[dr] & & *{2} \ar[ur] \ar[dr] & & *{} & & & & & & & \\
    *{} \ar@{-}[rrrr] \ar[uu] & & & & *{1} \ar[ur] \ar@{-}[rr] & & *{} \ar@{-}[rrrrrrrr] & & & & & & & & *{} \ar[uu]\\
    & & & & & \lefteqn{\textstyle{v\frac{n+1}{2}+1}} & & & & & & & & & \\
                      }
\]
in particular, $\sT_{\BX}$ is connected and has only a finite number
of isomorphism classes of indecomposable objects.  This shows that
$\sT_{\BX}$ satisfies the conditions of Theorem \ref{thm:Amiot}.

Now note that
\[
  \textstyle  v\frac{n+1}{2} + 1 = u\frac{m+1}{2} + 1
\]
because $v(n+1) = u(m+1)$.  Since $u$ is even, this implies that the
AR quiver of $\sT_{\BX}$ is precisely the AR quiver of the $u$-cluster
category of type $A_m$.  Hence Theorem \ref{thm:Amiot} with $\sU$
equal to the $u$-cluster category of type $A_m$ and $\sV$ equal to
$\sT_{\BX}$ says that these two categories are triangulated
equivalent.

In other words, the $u$-cluster category of type $A_m$ is triangulated
e\-qui\-va\-lent to $\sT_{\BX}$ which is a quotient category of $\sT$,
the $v$-cluster category of type $A_n$; this is the desired result.

Next the proof for $u$ and $v$ odd.  Here the AR quiver of $\sT$ is
given by diagram \eqref{equ:Aodd}.  The equation $v(n+1) = u(m+1)$
forces the difference $n-m$ to be even.  Let $\BX$ be add of the
indecomposable objects in two bands $\frac{n-m}{2}$ vertices wide
along the top and bottom of the AR quiver.  Then arguments like the
ones above show that the $u$-cluster category of type $A_m$ is
triangulated equivalent to the quotient category $\sT_{\BX}$, again as
desired.
\end{proof}

\begin{Corollary}
\label{cor:A}
Each $u$-cluster category of type $A$ is triangulated e\-qui\-va\-lent
to a quotient of a $1$- or a $2$-cluster category of type $A$.
\end{Corollary}

\begin{proof}
This is clear from Theorem \ref{thm:A}: If $u$ is odd then set $v =
1$ and $n = u(m+1) - 1$, and if $u$ is even then set $v = 2$ and $n
= \frac{u}{2}(m+1) - 1$.
\end{proof}

The corollary shows that in a sense, in type $A$, every phenomenon of
$u$-cluster categories is already implicit in $1$- and $2$-cluster
categories.  Both $1$- and $2$-cluster categories are needed since the
AR quiver of a $1$-cluster category is a M\"{o}bius band while that of
a $2$-cluster category is a wreath, whence neither can be obtained
from the other.

\setcounter{subsection}{3}
\subsection{Type D}

The AR quiver of $\D(kD_n)$ is $\BZ D_n$ by \cite[cor.\ 4.5]{Happel};
let me visualize this as a horizontal band,
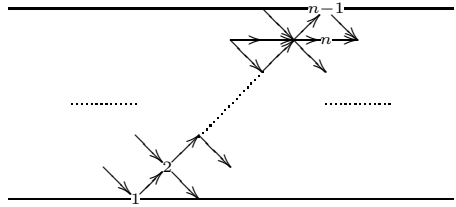
\[
  \def\labelstyle{\textstyle}
  \def\objectstyle{\scriptscriptstyle}
  \vcenter{
  \xymatrix @!0 @-1pc {
    *{} \ar@{-}[rrrrrrrr] & & & & & & & & *{} \ar[dr] \ar@{-}[rr] & & *{n-1} \ar[dr] \ar@{-}[rrrr] & & & & *{} \\
    & & & & & & & *{} \ar[dr] \ar[r] & *{} \ar[r] & *{} \ar[ur] \ar[dr] \ar[r] & *{n} \ar[r] & *{} & & & \\
    *{} & & & & & & & & *{} \ar[ur] & & *{} & & & & *{} \\
    & & *{} \ar@{.}[rr] & & *{} & & & & & & *{} \ar@{.}[rr] & & *{} & & &\\
    *{} & & & & *{} \ar[dr] & & *{} \ar@{.}[uurr] \ar[dr] & & & & & && & \\
    & & & *{} \ar[dr] & & *{2} \ar[ur] \ar[dr] & & *{} & & & & & & & \\
    *{} \ar@{-}[rrrr] & & & & *{1} \ar[ur] \ar@{-}[rr] & & *{} \ar@{-}[rrrrrrrr] & & & & & & & & *{} \\
                      }
          }.
\]
The action of $\tau^{-1}$ on the AR quiver is to shift one unit to the
right.  The action of $\Sigma$ is to shift $n-1$ units to the right,
and switch the two `exceptional' vertices $n-1$ and $n$ at the top if
$n$ is odd.  Both claims again follow from \cite[table p.\ 
359]{MiyachiYekutieli}.  Hence the action of $\tau^{-1}\Sigma^v$ on
the AR quiver is to shift $v(n-1) + 1$ units to the right, and switch
the exceptional vertices if $v$ and $n$ are both odd.

Let $\sT$ denote the $v$-cluster category of type $D_n$, that is,
\[
  \sT = \D(kD_n)/\tau^{-1}\Sigma^v.
\]
The AR quiver of $\sT$ is the AR quiver of $\D(kD_n)$ modulo the
action of $\tau^{-1}\Sigma^v$ by \cite[prop.\ 1.3]{BMRRT}.  Hence the
AR quiver of $\sT$ is $\BZ D_n$ modulo an automorphism given by
shifting $v(n-1) + 1$ units to the right, and switching the
exceptional vertices if $v$ and $n$ are both odd.

This means that the AR quiver of $\sT$ has the form 
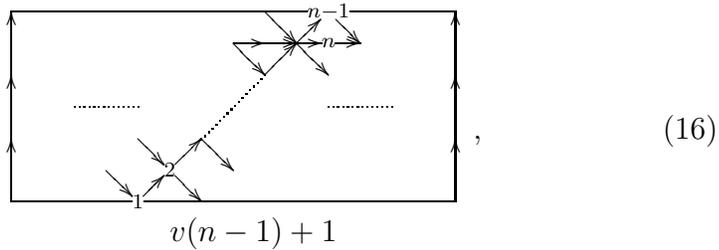
\begin{equation}
\label{equ:D}
  \def\labelstyle{\textstyle}
  \def\objectstyle{\scriptstyle}
  \vcenter{
  \xymatrix @!0 @-1pc {
    *{} \ar@{-}[rrrrrrrr] & & & & & & & & *{} \ar[dr] \ar@{-}[rr] & & *{n-1} \ar[dr] \ar@{-}[rrrr] & & & & *{} \\
    & & & & & & & *{} \ar[dr] \ar[r] & *{} \ar[r] & *{} \ar[ur] \ar[dr] \ar[r] & *{n} \ar[r] & *{} & & & \\
    *{} \ar[uu] & & & & & & & & *{} \ar[ur] & & *{} & & & & *{} \ar[uu] \\
    & & *{} \ar@{.}[rr] & & *{} & & & & & & *{} \ar@{.}[rr] & & *{} & & \\
    *{} \ar[uu] & & & & *{} \ar[dr] & & *{} \ar@{.}[uurr] \ar[dr] & & & & & & & & *{}\ar[uu]\\
    & & & *{} \ar[dr] & & *{2} \ar[ur] \ar[dr] & & *{} & & & & & & & \\
    *{} \ar@{-}[rrrr] \ar[uu] & & & & *{1} \ar[ur] \ar@{-}[rr] & & *{} \ar@{-}[rrrrrrrr] & & & & & & & & *{} \ar[uu]\\
    & & & & & \lefteqn{\textstyle{v(n-1)+1}} & & & & & & & & & \\
                      }
          }\;,
\end{equation}
where the vertical ends of the rec\-tan\-gle are identified.  The
number below the quiver indicates side length. 

The action of the AR translation on the quiver requires an
explanation: If $v$ and $n$ are both odd, then $\tau^{-1}\Sigma^v$
switches the exceptional vertices; otherwise, it does not.  This does
not make any difference to the way the quiver looks because the
exceptional vertices are attached to the rest of the quiver in a
symmetrical way.  But it does mean that in the two cases, the action
of $\tau^{-1}$ on $\BZ D_n$ induces different actions of $\tau^{-1}$
on the quiver \eqref{equ:D}.  Namely, if $v$ and $n$ are both odd so
the exceptional vertices are switched, then $\tau^{-1}$ also switches
between these vertices at the vertical ends of the rectangle, and so
$\tau^{-1}$ has {\em one long} orbit of exceptional vertices.  But if
$v$ or $n$ is even so the exceptional vertices are not switched, then
$\tau^{-1}$ does not switch between these vertices at the vertical
ends of the rectangle, and so $\tau^{-1}$ has {\em two shorter} orbits
of exceptional vertices.

By considering the indecomposable objects in a band $n-m$ vertices
wide along the bottom of the AR quiver and following the same strategy
as the proof of Theorem \ref{thm:A}, the following theorem can now be
proved.

\begin{Theorem}
\label{thm:D}
Let $u \geq v$ be positive integers and let $m,n \geq 4$ be integers
such that
\[
  u(m-1) = v(n-1).
\]
Suppose that if $u,m$ are both odd, then $v,n$ are both odd.

Then the $u$-cluster category of type $D_m$ is triangulated equivalent
to a quotient category of the $v$-cluster category of type $D_n$.
\end{Theorem}

\begin{Corollary}
\label{cor:D}
Each $u$-cluster category of type $D$ is triangulated e\-qui\-va\-lent
to a quotient of a $1$- or a $2$-cluster category of type $D$.
\end{Corollary}

\begin{proof}
This follows from Theorem \ref{thm:D}: If $u$ is odd then set $v =
1$ and $n = u(m-1) + 1$, and if $u$ is even then set $v = 2$ and $n
= \frac{u}{2}(m-1) + 1$.
\end{proof}

The corollary shows that in type $D$, just as in type $A$, every
phe\-no\-me\-non of $v$-cluster categories is implicit in $1$- and
$2$-cluster categories.  Both $1$- and $2$-cluster categories are
again needed because of the different actions of the AR translation
on the exceptional vertices.

\setcounter{subsection}{4}
\subsection{Type E}

The AR quiver of $\D(kE_n)$ is $\BZ E_n$ by \cite[cor.\
4.5]{Happel}.  Let me use $E_6$ for illustrative purposes,
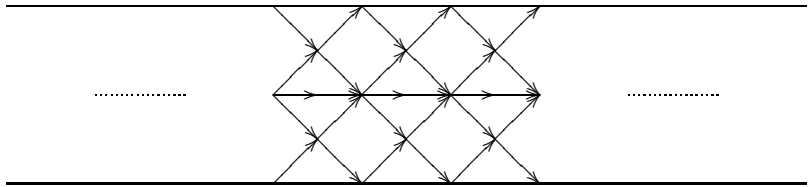
\[
  \def\labelstyle{\textstyle}
  \def\objectstyle{\scriptscriptstyle}
  \vcenter{
  \xymatrix @!0 @-0.6pc {
    *{} \ar@{-}[rrrrrrrrrrrrrrrrrr] & & & & & & *{} \ar[dr] & & *{} \ar[dr] & & *{} \ar[dr] & & *{} & & & & &&*{}\\
    & & & & & & & *{} \ar[dr] \ar[ur] & & *{} \ar[dr] \ar[ur] & & *{} \ar[dr] \ar[ur] & & & & &&&\\
    & & *{} \ar@{.}[rr] & & *{} & & *{} \ar[dr] \ar[ur] \ar[r] & *{} \ar[r] & *{} \ar[dr] \ar[ur] \ar[r] & *{} \ar[r] & *{} \ar[dr] \ar[ur] \ar[r] & *{} \ar[r] & *{} & &*{} \ar@{.}[rr] & & *{} &&\\
    & & & & & & & *{} \ar[dr] \ar[ur] & & *{} \ar[dr] \ar[ur] & & *{} \ar[dr] \ar[ur] & & & & &&&\\
    *{}\ar@{-}[rrrrrrrrrrrrrrrrrr] & & & & & & *{} \ar[ur] & & *{} \ar[ur] & & *{} \ar[ur] & & *{} & & & &&&*{}\\
                      }
          }.
\]
For both $E_6$, $E_7$, and $E_8$, the action of $\tau^{-1}$ on the AR
quiver is to shift one unit to the right.  For $E_6$, the action of
$\Sigma$ is to shift $6$ units to the right, and reflect in the
horizontal centre line.  For $E_7$, the action of $\Sigma$ is to shift
$9$ units to the right, and for $E_8$, the action is to shift $15$
units to the right.  See \cite[table p.\ 359]{MiyachiYekutieli}.
Hence the action of $\tau^{-1}\Sigma^v$ on the AR quiver is the
following: For $E_6$, it shifts $6v+1$ units to the right, and
reflects in the horizontal centre line if $v$ is odd.  For $E_7$, it
shifts $9v+1$ units to the right, and for $E_8$, it shifts $15v+1$
units to the right.

Just as in the previous sections, this permits me to compute the AR
quiver of $\D(kE_n)/\tau^{-1}\Sigma^v$, the $v$-cluster category of
type $E_n$.  For instance, in the case of $E_6$, the quiver is 
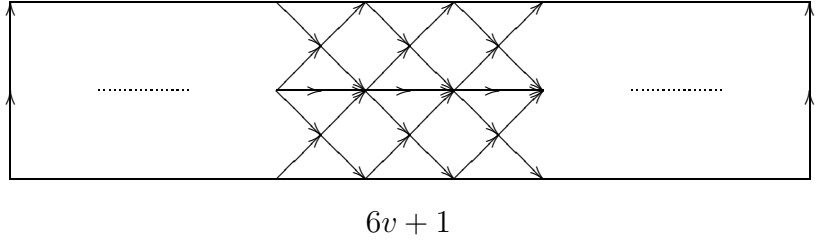
\[
  \def\labelstyle{\textstyle}
  \def\objectstyle{\scriptscriptstyle}
  \vcenter{
  \xymatrix @!0 @-0.6pc {
    *{} \ar@{-}[rrrrrrrrrrrrrrrrrr] & & & & & & *{} \ar[dr] & & *{} \ar[dr] & & *{} \ar[dr] & & *{} & & & & &&*{}\\
    & & & & & & & *{} \ar[dr] \ar[ur] & & *{} \ar[dr] \ar[ur] & & *{} \ar[dr] \ar[ur] & & & & &&&\\
    *{} \ar[uu] & & *{} \ar@{.}[rr] & & *{} & & *{} \ar[dr] \ar[ur] \ar[r] & *{} \ar[r] & *{} \ar[dr] \ar[ur] \ar[r] & *{} \ar[r] & *{} \ar[dr] \ar[ur] \ar[r] & *{} \ar[r] & *{} & &*{} \ar@{.}[rr] & & *{} &&*{}\ar[uu]\\
    & & & & & & & *{} \ar[dr] \ar[ur] & & *{} \ar[dr] \ar[ur] & & *{} \ar[dr] \ar[ur] & & & & &&&\\
    *{} \ar@{-}[rrrrrrrrrrrrrrrrrr] \ar[uu] & & & & & & *{} \ar[ur] & & *{} \ar[ur] & & *{} \ar[ur] & & *{} & & & &&&*{}\ar[uu]\\
    & & & & & & & & \lefteqn{\textstyle{6v+1}} & & & & & \\
                      }
          }
\]
if $v$ is even and
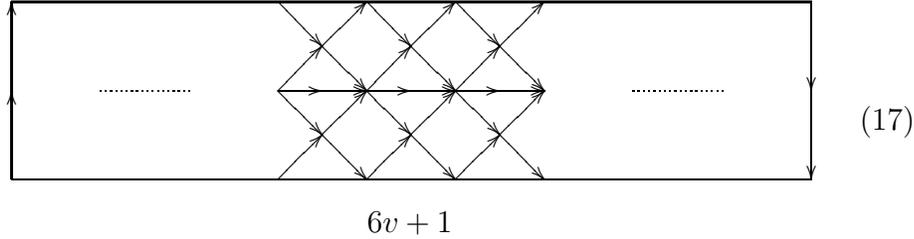
\begin{equation}
\label{equ:E6mobius}
  \def\labelstyle{\textstyle}
  \def\objectstyle{\scriptscriptstyle}
  \vcenter{
  \xymatrix @!0 @-0.6pc {
    *{} \ar@{-}[rrrrrrrrrrrrrrrrrr] & & & & & & *{} \ar[dr] & & *{} \ar[dr] & & *{} \ar[dr] & & *{} & & & & &&*{}\ar[dd]\\
    & & & & & & & *{} \ar[dr] \ar[ur] & & *{} \ar[dr] \ar[ur] & & *{} \ar[dr] \ar[ur] & & & & &&&\\
    *{} \ar[uu] & & *{} \ar@{.}[rr] & & *{} & & *{} \ar[dr] \ar[ur] \ar[r] & *{} \ar[r] & *{} \ar[dr] \ar[ur] \ar[r] & *{} \ar[r] & *{} \ar[dr] \ar[ur] \ar[r] & *{} \ar[r] & *{} & &*{} \ar@{.}[rr] & & *{} &&*{}\ar[dd]\\
    & & & & & & & *{} \ar[dr] \ar[ur] & & *{} \ar[dr] \ar[ur] & & *{} \ar[dr] \ar[ur] & & & & &&&\\
    *{} \ar@{-}[rrrrrrrrrrrrrrrrrr] \ar[uu] & & & & & & *{} \ar[ur] & & *{} \ar[ur] & & *{} \ar[ur] & & *{} & & & &&&*{}\\
    & & & & & & & & \lefteqn{\textstyle{6v+1}} & & & & & \\
                      }
          }
\end{equation}
if $v$ is odd.  In the cases of $E_7$ and $E_8$, the ends of the
rectangle always have the same orientations, and the horizontal
lengths are $9v+1$ and $15v+1$, respectively.

By deleting one or two rows of vertices from the AR quiver in type
$E_8$ or one row of vertices from the AR quiver in type $E_7$, I can
get down to types $E_6$ and $E_7$.  The resulting theorem is the
following.

\begin{Theorem}
\label{thm:E}
\begin{enumerate}
  
  \item Let $u$ and $v$ be positive integers with $3u = 5v$.  Then the
  $u$-cluster category of type $E_7$ is triangulated equivalent to a
  quotient category of the $v$-cluster category of type $E_8$.

\smallskip

  \item Let $u$ and $v$ be positive integers with $u$ even and $2u =
  5v$.  Then the $u$-cluster category of type $E_6$ is triangulated
  equivalent to a quotient category of the $v$-cluster category of
  type $E_8$.

\smallskip

  \item Let $u$ and $v$ be positive integers with $u$ even and $2u =
  3v$.  Then the $u$-cluster category of type $E_6$ is triangulated
  equivalent to a quotient category of the $v$-cluster category of
  type $E_7$.

\end{enumerate}
\end{Theorem}
In parts (ii) and (iii) of the theorem, $u$ is required to be even to
avoid that the AR quiver of the $u$-cluster category of type $E_6$ is
the M\"{o}bius band of figure \eqref{equ:E6mobius}.

\begin{Remark}
\label{rmk:E}
In contrast to the situation in types $A$ and $D$, Theorem
\ref{thm:E} does not provide for arbitrary $u$-cluster categories of
types $E_6$ and $E_7$ to be quotients of the $1$- and $2$-cluster
categories of type $E_8$.

Indeed, the method used in types $A$ and $D$ was to take a $1$- or a
$2$-cluster category of type $A_n$ or $D_n$ for some large $n$, then
trim its AR quiver to a smaller width.  The same idea cannot work in
type $E$, because the AR quivers of $v$-cluster categories of type $E$
have a fixed, small width.

It appears that in type $E$, phenomena of general $v$-cluster
categories are not implicit in the $1$- and $2$-cluster situations.
\end{Remark}

\setcounter{subsection}{25}
\subsection{Mixed type}

By employing different deletions of vertices, it is also possible to
go from type $D_n$ to type $A_m$ for $m < n$, and from types $E_6$,
$E_7$, $E_8$ to types $A_2$ through $A_7$ and types $D_4$ through
$D_7$.  The details of this are left to the reader.

\medskip
\noindent
{\em Acknowledgement. }  I would like to thank Claire Amiot for
answering several questions on her paper \cite{Amiot}.

\end{document}